\documentclass[preprint]{elsarticle}
\journal{Journal of Number Theory}
\usepackage{amsfonts, amssymb, amsmath, amsthm}
\usepackage{url, centernot}
\usepackage{ctable}

\newtheorem{thm}{Theorem}

\newtheorem{prop}{Proposition}

\newtheorem{cor}{Corollary}

\newtheorem{lem}{Lemma}

\theoremstyle{definition}
\newtheorem*{defn}{Definition}

\newcommand{\thmref}[1]{Theorem \ref{#1}}
\newcommand{\propref}[1]{Proposition \ref{#1}}
\newcommand{\cororef}[1]{Corollary \ref{#1}}
\newcommand{\lemref}[1]{Lemma \ref{#1}}

\newcommand{\thickmidrule}{\midrule[\heavyrulewidth]}

\renewcommand{\emptyset}{\ensuremath{\varnothing}}
\renewcommand{\phi}{\varphi}
\makeatletter
\newcommand{\imod}[1]{\allowbreak\mkern10mu({\operator@font mod}\,\,#1)}
\newcommand{\jmod}[1]{\allowbreak\mkern6mu({\operator@font mod}\,\,#1)}
\makeatother
\newcommand{\deq}{\mathrel{\mathop:}=}
\newcommand{\divides}{\mathrel{\mid}}
\newcommand{\notdivides}{\mathrel{\centernot\mid}}

\newcommand{\Z}{\ensuremath{\mathbb{Z}}}
\newcommand{\N}{\ensuremath{\mathbb{N}}}
\newcommand{\Q}{\ensuremath{\mathbb{Q}}}
\newcommand{\R}{\ensuremath{\mathbb{R}}}
\newcommand{\C}{\ensuremath{\mathbb{C}}}

\renewcommand{\Re}{\mathop{\mathrm{Re}}}
\renewcommand{\Im}{\mathop{\mathrm{Im}}}
\DeclareMathOperator{\sign}{sign}
\DeclareMathOperator{\Res}{Res}
\DeclareMathOperator{\Log}{Log}

\newcommand{\Li}[1]{\mathop{\mathrm{Li}_{#1}}}
\newcommand{\RR}[1]{\mathop{\mathfrak{R}_{#1}}}
\newcommand{\logp}{\mathop{\log^+\!}}
\renewcommand{\O}{\mathop{O}_{}}
\newcommand{\Ok}{\mathop{O_k}}
\newcommand{\Osub}[1]{\mathop{O_{#1}}}
\renewcommand{\o}{\mathop{o}}
\newcommand{\tildeD}[1]{\mathop{\tilde{D}^{#1}}}
\newcommand{\h}[2]{\mathop{H^{#1}_{#2}}\!}
\providecommand{\abs}[1]{\lvert#1\rvert}
\providecommand{\norm}[1]{\lVert#1\rVert}

\newcommand{\struta}[1]{\rule{0pt}{#1}}
\newcommand{\strutb}[1]{\rule[-#1]{0pt}{0.5mm}}
\newcommand{\eval}[2]{\left[#1\right]_{\partial #2}}
\newcommand{\stirone}[2]{\Bigl[\genfrac{}{}{0pt}{}{#1}{#2}\Bigr]}
\newcommand{\stirtwo}[2]{\Bigl\{\genfrac{}{}{0pt}{}{#1}{#2}\Bigr\}}
\newcommand{\stironesm}[2]{{\textstyle \genfrac{[}{]}{0pt}{}{#1}{#2}}}
\newcommand{\stirtwosm}[2]{{\textstyle \genfrac{\{}{\}}{0pt}{}{#1}{#2}}}
\newcommand{\quasi}{\overset{*}{\sim}}
\newcommand{\UC}{\mathbb{T}}
\newcommand{\Pyinline}{{\partial P/\partial y}}

\newcommand{\rtfnset}{R}
\newcommand{\rtfnsettor}{\rtfnset_{\mathrm{tor}}}
\newcommand{\rtfnsetnon}{\rtfnset_{\mathrm{non}}}
\newcommand{\ipm}{I^{\pm}_{n,m}}
\newcommand{\jpmk}{J^{\pm}_{n,m,k}}
\newcommand{\jpmzero}{J^{\pm}_{n,m,0}}
\newcommand{\Spm}{S^{\pm}_{n,m,r}}
\newcommand{\Spmone}{S^{\pm}_{n,m,1}}
\newcommand{\slicepoint}{\alpha_{0}}

\begin{document}

\begin{frontmatter}

\title{Asymptotic expansion of the difference of two Mahler measures}

\author{John D.~Condon \fnref{currentaddress}}


\address{Department of Mathematics \\
	Amherst College\\
        Amherst, MA 01002 USA}

\fntext[currentaddress]{Current address: 807 Buchanan St NW, Washington, DC 20011, USA}

\ead{jconecker@gmail.com}

\ead[url]{http://jconecker.wordpress.com}

\begin{abstract}
We show that for almost every polynomial $P(x,y)$ with complex coefficients, the difference of the logarithmic Mahler measures of $P(x,y)$ and $P(x,x^n)$ can be expanded in a type of formal series similar to an asymptotic power series expansion in powers of $1/n$.  This generalizes a result of Boyd.  We also show that such an expansion is unique and provide a formula for its coefficients.  When $P$ has algebraic coefficients, the coefficients in the expansion are linear combinations of polylogarithms of algebraic numbers, with algebraic coefficients.
\end{abstract}


\end{frontmatter}



\section{Introduction}\label{sec:intro}

For a nonzero Laurent polynomial $P\in\C\bigl[x_1^{\pm 1},\dotsc x_n^{\pm 1}\bigr]$, the \emph{(logarithmic) Mahler measure} of $P$ is defined as
\begin{equation}\label{eq:multivar}
  \begin{split}
  m(P) &= \int_{0}^{1}\dotsi \int_{0}^{1}\log\bigl|P \bigl(\exp(2\pi i t_{1}),\ldots,\exp(2\pi i t_{n})\bigr)\bigr|
        	    \, d t_{1}\cdots d t_{n} \\[2mm]
        &= \frac{1}{(2\pi i)^n}\int_{\UC^n}\log\bigl|P(x_1,\ldots,x_n)\bigr|
            \,\frac{dx_1}{x_1}\cdots\frac{dx_n}{x_n},
   \end{split}
\end{equation}
where $\UC$ is the unit circle in $\C$, oriented counter-clockwise.  This integral is always finite, even if the zero set of $P$ intersects $\UC^n$.  

When $n=1$, Jensen's formula implies that if $P(x)=a\prod_{j=1}^d (x-\alpha_j)$, then
\begin{equation}\label{eq:onevar}
  m(P) = \log(a) + \sum_{j=1}^{d} \logp \abs{\alpha_{j}},
\end{equation}
where, for $r>0$, $\logp(r)\deq \log\bigl(\max\{r, 1\}\bigr)$.
The latter construct (or actually, its exponential) was first studied by D. H. Lehmer \cite{lehmer} in the 1930s.  Mahler introduced \eqref{eq:multivar} three decades later \cite{mahler}.

Boyd \cite{boyd2} established the following connection (generalized by Lawton \cite{lawton}) between multivariable and single-variable Mahler measure values.

\begin{thm}[Boyd]\label{thm:boyd}
  For any nonzero Laurent polynomial $P(x,y)$ with complex coefficients,
\[
m\bigl(P(x,y)\bigr) = \lim_{n\to\infty} m\bigl(P(x,x^n)\bigr).
\]
\end{thm}

Also in \cite{boyd2}, Boyd proved the following result, which shows the rate at which the above limit converges in the case of $P(x,y)=1+x+y$:
\begin{prop}[Boyd]\label{prop:boydasymp}
For all positive integers $n$,
\begin{equation}\label{eq:boydexpn}
  m(1+x+x^n) - m(1+x+y) = \frac{c(n)}{n^2} + \O\biggl(\frac{1}{n^3}\biggr),
\end{equation}
where $c(n)$ depends only on $n$ mod 3:
\[
c(n)=\begin{cases}
    \phantom{-}\sqrt{3} \pi/18 & \mathrm{if}\ n\equiv 0,1 \imod{3} \\[2 mm]
    -\sqrt{3} \pi/6 & \mathrm{if}\ n\equiv 2 \imod{3}. \\[2 mm]
    \end{cases}
\]
\end{prop}

Motivated by these results, we examine the difference between these Mahler measures.
\begin{defn}
For a nonzero Laurent polynomial $P(x,y)$ and a positive integer $n$, let $\Delta_{n}(P) \deq m\bigl(P(x,x^n)\bigr)-m\bigl(P(x,y)\bigr)$.
\end{defn}

The right side of \eqref{eq:boydexpn} could be thought of as the beginning of a formal series for $\Delta_{n}(1+x+y)$ of the form $\sum_{k=2}^\infty c_k(n)/n^k$.  We will find such an expression for $\Delta_{n}(P)$, for $P=1+x+y$ as well as many other two-variable polynomials.

Such a formal series cannot quite be called an asymptotic power series in $n$, in the sense of \cite{erdelyi}, in that the coefficients in such a series should be independent of $n$.  But our coefficients will have a structure that will, in particular, make them bounded as functions of $n$, occasionally depending only on $n \jmod{m}$ for some integer $m$.


\section{Statement of results}

Unless stated otherwise, all variables and functions are complex-valued.  $\N$ will denote the set of positive integers.  For $P\in\C[z_1,\ldots,z_n]$, let $Z(P)$ denote the affine zero set of $P$ in $\C^{n}$.

$\Li{k}(z)$ will denote the principal branch of the $k$-th polylogarithm function \cite{lewin}.  For $k\ge 2$ (which is all we will need) and $\abs{z}\le 1$, this is given by
\[
\Li{k}(z)  = \sum_{n=1}^{\infty} \frac{z^n}{n^k}.
\]

\begin{defn}
We will say that a function $\omega:\R\to\R$ is \emph{quasiperiodic} if it is the sum of finitely many continuous, periodic functions.
\end{defn}

Quasiperiodic functions are clearly bounded (although this is no longer true if the summand functions are not assumed to be continuous \cite{keleti}).

\begin{defn}
For a function $f:\N\to\R$, we will say that a formal series $\sum_{r=0}^\infty c_r(n)/n^r$ is an \emph{asymptotic pseudo-power series} (or \emph{a.p.p.s.}) \emph{expansion} of $f(n)$ (in powers of $1/n$, as $n\to\infty$) if, for each nonnegative integer $r$, $c_r(n)$ is the restriction to the nonnegative integers of a quasiperiodic function of $n$, and if for all postive integers $n$ and $k$,
\[
f(n)=\sum_{r=0}^{k-1} \frac{c_r(n)}{n^r} + \Osub{f,k}\biggl(\frac{1}{n^{k}}\biggr).
\]
(The subscripts on ``$\O$'' indicate that the implied constant depends on those subscripts).
We will denote this by writing $f(n) \quasi \sum_{r=0}^{\infty} c_r(n)/n^r$.
%
\end{defn}

Asymptotic power series expansions in powers of $1/n$, as defined in \cite{erdelyi}, are the same as a.p.p.s.\ expansions in which the coefficients $c_r(n)$ do not depend on $n$.  We will refer to these as \emph{true} asymptotic power series expansions, for contrast.


True asymptotic power series expansions for a function are uniquely determined by that function.  For our series, quasiperiodicity of the coefficients is enough to rescue uniqueness.
\begin{prop}\label{prop:uniqueness}
  If a function $f:\N\to\R$ has an a.p.p.s.\ expansion in powers of $1/n$, that expansion is unique.
\end{prop}

The following is our main result.

\begin{thm}\label{thm:main}
  Let $P(x,y)\in\C[x,y]$ be such that $P$ and $\Pyinline$ do not have a common zero on $\UC\times\C$.  Then $\Delta_{n}(P) = m\bigl(P(x,x^n)\bigr)-m\bigl(P(x,y)\bigr)$ has an a.p.p.s.\ expansion in powers of $1/n$.
Specifically, if $P$ is absolutely irreducible, then $\Delta_{n}(P) \quasi \sum_{r=2}^{\infty} c_r(n)/n^r$, with coefficients
\begin{equation}\label{eq:maincoeffformula}
  c_r(n)
    = \frac{1}{\pi}\sum_{(\alpha,\beta)\in \mathcal{E}} s(\alpha,\beta)
    		\sum_{a=2}^{r} \RR{a+1}\bigl(\Omega_{P,r,a}(\alpha,\beta)\bigr) \RR{a}\bigl(\Li{a}(\beta/\alpha^n)\bigr),
\end{equation}
where:
\begin{itemize}
  \item $\mathcal{E} = Z(P)\cap \UC^{2}$, if this set is finite, or $\emptyset$ if it is infinite,
  \item $s(x,y)$ is a function from $\mathcal{E}\to \{-1,0,1\}$, independent of $r$,
  \item $\RR{a}(z)=\Re(z)$ when $a$ is even and $\Im(z)$ when $a$ is odd,
  \item $\Omega_{P,r,a}(x,y)$ is a rational function, obtained by taking an integer polynomial $\Psi_{r,a}$ (independent of $P$) in indeterminates $w_{i,j}$, and making the substitutions
	$w_{i,j} = x^{i} y^{j-1} \bigl(\frac{\partial^{i+j}}{\partial x^{i} \partial y^{j}} P\bigr)/\bigl(\frac{\partial}{\partial y}P\bigr)$.
\end{itemize}
In particular, if $P$ has integer coefficients, then the coefficients $c_{r}(n)$ are $\overline{\Q}$-linear combination of polylogarithms evaluated at algebraic arguments.
\end{thm}
(The definition of $s(x,y)$ is given in section~\ref{sec:proofofmain}, and $\Psi_{r,a}$ is defined in \eqref{eq:Psidef}.)

The only appearance of $n$ in the formula for $c_{r}(n)$ is in $\Li{a}(\beta/\alpha^n)$.  Hence, if the $x$-coordinates of the points in $\mathcal{E}$ are all $m$-th roots of unity, then for each $r$, $c_{r}(n)$ depends only on $n \jmod{m}$.

Because the a.p.p.s.\ above only sums over $r\ge 2$, $\Delta_{n}(P) = \Osub{P}(n^{-2})$ for all $P$ as in the theorem.  This fact was proved by Boyd in \cite{boyd2}.  Indeed, he was able to go further, showing that for every nonzero $P\in\C[x,y]$, $\Delta_{n}(P) = \Osub{P}(n^{-(1+\delta)})$ for some $\delta>0$.

Of course, if $\mathcal{E}$ is empty, the coefficients $c_{r}(n)$ in the theorem are all trivially equal to $0$.  It turns out that there is a simple sufficient criterion for when this occurs.  Let $P^{\star}(x,y)$ be the reciprocal polynomial of $P$ (defined in section~\ref{sec:notation}).
\begin{cor}\label{cor:zeroseq}
Let $P(x,y)$ be as in \thmref{thm:main}.  If $P^{\star}=cP$ for some constant $c\in \C$, then $\mathcal{E} = \emptyset$, hence $\Delta_{n}(P) \quasi 0$.
\end{cor}


Boyd's \propref{prop:boydasymp} corresponds to $P(x,y)=1+x+y$.  We can be more explicit in that case:
\begin{cor}\label{cor:main}
$\Delta_{n}(1+x+y) \quasi \sum_{r=2}^{\infty} c_r(n)/n^r$, for
\begin{equation}\label{eq:maincoroformula}
c_r(n) =  \frac{2}{\pi} \!\!
	 \sum_{\substack{a,b \\ a+b=r+1}} 
	 	\!\! (-1)^{b} \RR{a}\bigl(\Li{a}(\xi^{n+1})\bigr)
		\sum_{j=b}^{r-1} \stirone{j}{b}\stirtwo{r-1}{j} \RR{a+1}(\xi^{j}),
\end{equation}
where $\stironesm{a}{b}$ and $\stirtwosm{a}{b}$ represent the (unsigned) Stirling numbers of the first and second kind, respectively, and $\xi=\exp(2\pi i/3)$.
In particular, the coefficients $c_{r}(n)$ are linear combinations of polylogarithms evaluated at third roots of unity, with coefficients in $\frac{1}{2\pi} \Z[\sqrt{3}]$.
\end{cor}

After some preliminary lemmas, we prove the above results in sections \ref{sec:proofofuniqueness} through \ref{sec:proofofcor}.  In section~\ref{sec:earlycoeffs}, we give expressions for the first two nonzero coefficients, $c_{2}(n)$ and $c_{3}(n)$, and verify that the value for $c_2(n)$ in \cororef{cor:main} agrees with Boyd's $c(n)$ from \propref{prop:boydasymp}.
In section~\ref{sec:numerics}, we presents the results of numerical calculations, giving support for our formulas and analyzing a degenerate case.  Section~\ref{sec:signdet} describes an algebraic procedure for determining the values of the function $s(x,y)$ mentioned in \thmref{thm:main}.


\section{Additional Notation and Definitions}\label{sec:notation}

For a function $f(z_1,\ldots,z_n)$,	let $\overline{f}(z_1,\ldots,z_n)\deq \overline{f(\overline{z_1},\ldots,
	\overline{z_n})}$.

For a nonzero polynomial $P(z_1,\ldots,z_n)$, its \emph{reciprocal polynomial} is defined as 
    $P^\star(z_1,\ldots,z_n) \deq z_1^{d_1} \cdots z_n^{d_n} \, \overline{P}(z_1^{-1},\ldots, z_n^{-1})$, where $d_i=\deg_{z_i}(P)$.  We will say that $P$ is \emph{almost self-reciprocal} (or \emph{a.s.r.}) if $P^{\star}=cP$ for some constant $c$.

For $Q_1, Q_2\in\C[x,y]$, thinking of them as as polynomials in $y$ with coefficients in $\C[x]$, we let $\Res_y(Q_1,Q_2)$ and $\delta_y(Q_1)$ denote, respectively, the resultant and discriminant with respect to $y$.  These are both polynomials in $x$ alone.

An \emph{arc} is, for us, a connected subset of $\UC$, including possibly $\UC$ itself.  We will say that $\UC$ is an improper arc, and all others are proper.  An \emph{open} (respectively, \emph{closed}) arc is an arc that is open (closed) relative to the usual topology on $\UC$.  Arcs will be oriented counter-clockwise.

For an open, proper arc $\gamma$, equal to $\bigl\{e^{it} | \alpha<t<\beta\bigr\}$ with $0<\beta-\alpha\le 2\pi$, let $\eval{f(x)}{\gamma} \deq \lim_{t\to 0^{+}} \bigl(f(e^{i(\beta-t)})-f(e^{i(\alpha+t)})\bigr)$.  (Usually, this is just the difference of the values of $f$ at the endpoints of $\gamma$.  The limit is most useful when $\beta = \alpha+2\pi$ and $f$ is discontinuous at $e^{i\alpha}=e^{i\beta}$.)

For a set $S\subseteq\C$, let  $S^{-1}\deq\{z^{-1}|z\in S\}$ and $\overline{S}\deq\{\overline{z}|z\in S\}$.

We will follow the convention of letting the  Pochhammer symbol $(x)_k$ denote the ``falling factorial'' $x(x-1)(x-2)\dotsm (x-k+1)= k! \binom{x}{k}$.

As mentioned earlier, $\stironesm{n}{k}$ and $\stirtwosm{n}{k}$ will represent the (unsigned) Stirling numbers of the first and second kind, respectively \cite{wilf}.  These are nonnegative integers, and are nonzero if and only if $1\le k\le n$, with the exception that $\stironesm{0}{0}=\stirtwosm{0}{0}=1.$ 

For integers $m$ and $k$ with $k\ge 0$, define an operator $H_{k}^{m}$ on $C^{k}$ functions $f(z)$ as follows:
\[
\h{m}{k}f \deq \Bigl(z \frac{d}{dz}\Bigr)^k \bigl(f(z)^m\bigr).
\]

Let $B_{n,k}(z_1,z_2,\ldots,z_{n-k+1})$ be the exponential, partial Bell polynomial
\[
\sum \frac{n!}{j_1 ! j_2! \cdots j_{n-k+1}!}
	\prod_{i=1}^{n-k+1} \Bigl(\frac{z_i}{i!}\Bigr)^{j_i},
\]
where the sum is taken over all tuples $(j_1,\dotsc, j_{n-k+1})$ of nonnegative integers such that $\sum_i j_i=k$ and $\sum_i i j_i=n$.  These polynomials have integer coefficients and are homogeneous of degree $k$
\cite{comtet1}.

For integers $n,k$, we define
\[
\Phi_{n,k}(y_0,\ldots,y_{n-k+1}) \deq
	\sum_{i, j} (-1)^{i-k} \stirone{i}{k} \stirtwo{n}{j} {y_0}^{n-i} B_{j,i}(y_1,\ldots,y_{j-i+1}).
\]
(The summands vanish except when $0\le k\le i\le j \le n$.)
This is clearly also an integer polynomial.  Like $B_{n,k}$, $\Phi_{n,k}$ is identically zero if $n<k$, $k<0$, or if $k=0$ and $n>0$; otherwise, $\Phi_{n,k}$ is homogeneous of degree $n$.  (See Table~\ref{PhiTable}.)  These polynomials become unwieldy for large $n$---for instance, $\Phi_{10,1}$ has 138 terms.

\begin{table}[bt]
  \centering
  \begin{tabular}{|@{\,}c@{\,}|@{\;}c@{\;}|p{1.55in} |p{1.1in}|p{.5in}|@{\;\,}c@{\;\,}|}
    \hline
      $n \,\backslash\, k$ & \strutb{2.2mm}0\struta{4.5mm} & 1 & 2 & 3 & 4 \\ \hline
      0 & \strutb{2.2mm} 1 \struta{4.5mm} & 0 & 0 & 0 & 0 \\ \hline
      1 & \strutb{2.2mm} 0 \struta{4.5mm} & $y_1$ & 0 & 0 & 0 \\ \hline
      2 & \strutb{2.2mm} 0 \struta{4.5mm} & $y_0 y_1 + y_0 y_2 - y_1^2$ & $y_1^2$ & 0 & 0\\ \hline
      3 & 0\struta{4.5mm} & $y_0^2 y_1 + 3 y_0^2 y_2 + y_0^2 y_3$
                    & $3 y_0 y_1^2 + 3 y_0 y_1 y_2$ & $y_1^3$ & 0 \\[1mm]
        & & $\,-\, 3 y_0 y_1^2 - \strutb{2.2mm} 3 y_0 y_1 y_2 \,+\, 2 y_1^3$
                    & $\,-\, 3 y_1^3$ &  &\\ \hline
      4 & 0\struta{4.5mm} & $y_0^3 y_1 + 7 y_0^3 y_2 + 6 y_0^3 y_3$
      	& $7 y_0^2 y_1^2 + 18 y_0^2 y_1 y_2$ & $6 y_0 y_1^2 y_2$ & $y_1^4$ \\[1mm]
	& & $\,+\, y_0^3 y_4 - 7 y_0^2 y_1^2 - 18 y_0^2 y_1 y_2$
		& $\,+\, 4 y_0^2 y_1 y_3 + 3 y_0^2 y_2^2$ & $\,+\, 6 y_0 y_1^3$ & \\[1mm]
	& & $\,-\, 4 y_0^2 y_1 y_3 - 3 y_0^2 y_2^2$
		& $\,-\, 18 y_0 y_1^2 y_2$ & $\,-\, 6 y_1^4$ & \\[1mm]
	& & $\,+\, 12 y_0 y_1^3 + 12 y_0 y_1^2 y_2 - 6 y_1^4$
		& $\,-\, 18 y_0 y_1^3 + \strutb{2.2mm} 11 y_1^4$ & & \\ \hline
  \end{tabular}
  \caption{The polynomials $\Phi_{n,k}$, for $0\le k, n\le 4$.}
  \label{PhiTable}
\end{table}

For each integer $n\ge 1$, we define an integer polynomial $Q_{n}$ in indeterminates $w_{i,j}$ (with $i,j\ge 0$ and $i+j\le n$) as follows.  Let $E_{n}$ be the set of all doubly-indexed sequences $\mathbf{e} = \{e_{i,j}\}_{i,j\ge 0}$ of nonnegative integers satisfying $e_{0,0}=0$, $\sum_{i,j} i e_{i,j} = n$, $\sum_{i,j} j e_{i,j} = 2n-2$, and $\sum_{i,j} e_{i,j} = 2n-1$.  (Each such sequence will clearly have only finitely many nonzero terms.)  For each $\mathbf{e}\in E_{n}$, let $w_{\mathbf{e}} \deq \prod_{i,j\ge 0} w_{i,j}^{e_{i,j}}$ and
\[
b_{n,\mathbf{e}} \deq (-1)^{2n-1-e_{0,1}}
	\frac{n! (2n-2-e_{0,1})! e_{0,1}!}{\prod_{i,j\ge 0} e_{i,j}! (i!j!)^{e_{i,j}}}.
\]
We then define $Q_{n} \deq \sum_{\mathbf{e}\in E_{n}} b_{n,\mathbf{e}} \,w_{\mathbf{e}}$.  (See Wilde \cite{wilde} for a proof that the coefficients $b_{n,\mathbf{e}}$ are integers.)
For example, $Q_{1} = -w_{1,0}$ and $Q_{2} = -w_{0,1}^{2} w_{2,0} + 2 w_{0,1} w_{1,0} w_{1,1} - w_{0,2} w_{1,0}^{2}$.

For integers $2\le a \le r$, let
\begin{equation}\label{eq:Psidef}
\Psi_{r,a} = \Phi_{r-1,r-a+1}\bigl(1, Q_{1}, Q_{2}, \ldots\bigr).
\end{equation}
This is an integer polynomial in $w_{i,j}$, for $i,j\ge 0$ and $i+j\le a-1$.  Some examples of $\Psi_{r,a}$ are shown in Table~\ref{PsiTable}.
 
\begin{table}[bt]
  \centering
  \begin{tabular}{|c|c|p{3.75in}|}
  \hline
  $r$	& $a$ & $\Psi_{r,a}$ \strutb{2.2mm}\struta{4.5mm} \\ \hline
  2	& 2	& $-w_{1,0}$ \strutb{2.2mm}\struta{4.5mm} \\ \hline
  3	& 2	& $w_{1,0}^{2}$ \strutb{2.2mm}\struta{4.5mm} \\  \cline{2-3}
  	& 3	& $-w_{0,1}^{2} w_{2,0} + 2 \, w_{0,1} w_{1,0} w_{1,1} 
			- w_{0,2} w_{1,0}^{2} - w_{1,0}^{2} - w_{1,0}$  
				\struta{4.5mm} \strutb{2.2mm} \\ \hline
  4	& 2	& $-w_{1,0}^{3}$ \strutb{2.2mm}\struta{4.5mm} \\  \cline{2-3}
  	& 3	& $3 \, w_{0,1}^{2} w_{1,0} w_{2,0} - 6 \, w_{0,1} w_{1,0}^{2} w_{1,1} 
  			+ 3 \, w_{0,2} w_{1,0}^{3} + 3 \, w_{1,0}^{3} + 3 \, w_{1,0}^{2}$
			\struta{4.5mm} \strutb{2.2mm} \\  \cline{2-3}
	& 4	& $-w_{0,1}^{4} w_{3,0} + 3 \, w_{0,1}^{3} w_{1,0} w_{2,1} 
			+ 3 \, w_{0,1}^{3} w_{1,1} w_{2,0} - 3 \, w_{0,1}^{2} w_{1,0}^{2} w_{1,2}$
			 \struta{4.5mm} \\
	&	& \hspace{0mm}  $ -\, 3 \, w_{0,1}^{2} w_{0,2} w_{1,0} w_{2,0} 
			- 6 \, w_{0,1}^{2} w_{1,0} w_{1,1}^{2}
			+ 9 \, w_{0,1} w_{0,2} w_{1,0}^{2} w_{1,1}$ \\
	&	& \hspace{0mm}  $+\, w_{0,1} w_{0,3} w_{1,0}^{3} - 3 \, w_{0,2}^{2} w_{1,0}^{3}
			- 3 \, w_{0,1}^{2} w_{1,0} w_{2,0} + 6 \, w_{0,1} w_{1,0}^{2} w_{1,1}$ \\
	&	& \hspace{0mm}  $-\, 3 \, w_{0,2} w_{1,0}^{3} - 3 \, w_{0,1}^{2} w_{2,0} 
			+ 6 \, w_{0,1} w_{1,0} w_{1,1} - 3 \, w_{0,2} w_{1,0}^{2}$ \\
	&	& \hspace{0mm}  $ -\, 2 \, w_{1,0}^{3} - 3 \, w_{1,0}^{2} - w_{1,0}$
			  \strutb{2.2mm} \\ \hline
  \end{tabular}
  \caption{The polynomials $\Psi_{r,a}$, for $2\le a, r\le 4$.}
  \label{PsiTable}
\end{table}

Finally, given a polynomial $P\in\C[x,y]$, an open set $U\subseteq \C$, and a continuous function $\rho:U\to\C$,  we will say that $\rho(x)$ is a \emph{root function} for $P(x,y)$ on $U$ if $P\bigl(x,\rho(x)\bigr)= 0$ for all $x\in U$.


\section{Lemmas}\label{sec:lemmas}

Let $P(x,y)\in\C[x,y]$.  Write
\[
P(x,y)=\sum_{j=0}^d a_j(x)y^j,
\]
where $d\deq \deg_y(P)$ and $a_j(x)\in\C[x]$ for all $j$.  Let $\mathcal{R}(x)\deq\Res_y\bigl(P,\Pyinline \bigr)$, which is also equal to $\pm a_d(x)\delta_y(P)$.  For the remainder of this article, we will assume (as in \thmref{thm:main}) that $P$ and $\Pyinline$ have no common root on $\UC\times\C$.  This is equivalent to $\mathcal{R}(x)$ not having any roots on $\UC$.  (The roots of $\mathcal{R}(x)$ are sometimes referred to as the \emph{critical points} of $P$.)

Also, for this section, we will assume that $P$ is absolutely irreducible, with $\deg_{x}(P)$ and $\deg_{y}(P) >0$.  It follows easily that $P^\star$ is also irreducible, with $\deg_x(P^\star)=\deg_x(P)$, $\deg_y(P^\star)=\deg_y(P)$, and $P^{\star \star}=P$.

The Implicit Function Theorem guarantees the existence of root functions for $P$, away from the roots of $\mathcal{R}(x)$.  More concretely, we have the following \cite[Section 45]{mark3}:
\begin{lem}\label{lem:algebraic}
For any open, simply connected set $U\subset \C\setminus Z(\mathcal{R})$, there exist exactly $d$ root functions $\rho_1(x),\ldots,\rho_d(x)$ of $P(x,y)$ on $U$.  These are all holomorphic, single-valued functions, and for distinct $i, j$, there is no $\alpha\in U$ for which $\rho_{i}(\alpha)=\rho_{j}(\alpha)$.
Furthermore, for all $x\in U$, we have the factorization
\[
P(x,y)=a_d(x)\prod_{j=1}^d \bigl(y-\rho_j(x)\bigr).
\]
\end{lem}


We now turn to some applications of the reciprocal polynomial, $P^{\star}(x,y)$.
Its main utility for us stems from the following two observations: if $P(\alpha,\beta)=0$, then $P^{\star}(1/\overline{\alpha}, 1/\overline{\beta})=0$ as well; and if $\abs{z} = 1$, then $z= 1/\overline{z}$.  Therefore, $Z(P)\cap \UC^{2} \subseteq Z(P)\cap Z(P^{\star})$.

\begin{lem}\label{lem:betarootfunction}
Let $U\subset\C^{\times}$, and suppose that $\rho(x)$ is a root function for $P(x,y)$ on $U$.  Let $S$ be the set of all zeros of $\rho(x)$ on $U$.  Then the function $\rho^{\star}(x) \deq 1/\overline{\rho}(1/x)$ is a root function for $P^{\star}$ on $\overline{(U\setminus S)}^{\,-1}$.
\end{lem}

\begin{proof}
For all $x\in\overline{(U\setminus S)}^{\,-1}$, $w\deq 1/\overline{x}\in U\setminus S$, and $\overline{\rho}(1/x) = \overline{\rho(w)} \ne 0$, hence $\rho^{\star}(x)$ is defined.
Since $w\in U$, $P(w,\rho(w))=0$, hence $0= P^{\star}\bigl(1/\overline{w},1/\overline{\rho}(\overline{w})\bigr) = P^{\star}\bigl(x,\rho^{\star}(x)\bigr)$.
\end{proof}

If we restrict to $x\in\UC$, then $\arg \rho^{\star}(x) = \arg \rho(x)$, and $\abs{\rho^{\star}(x)} = 1/\abs{\rho(x)}$.


\begin{lem}\label{lem:essentialset}
  Let $\gamma$ be a proper, open arc in $\UC$, and $U$ a simply connected, open subset of $\C\setminus Z(\mathcal{R})$ such that $\gamma = U\cap \UC$.  Let $\rho(x)$ be a root function for $P(x,y)$ on $U$.
\begin{enumerate}
	\item If there exist infinitely many points $x\in \gamma$ such that $\abs{\rho(x)}=1$, then $P$ must be a.s.r.
	\item If $P$ is a.s.r.\ and there exists an $\alpha\in\gamma$ such that $\abs{\rho(\alpha)}=1$, then $\abs{\rho(x)}\equiv1$ for all $x\in\gamma$.
\end{enumerate}
\end{lem}

\begin{proof}
Let $T\deq \rho^{-1}(\UC) \cap\gamma$, let $S$ be the set of all roots of $\rho(x)$ on $U$, and let $V$ be an open, simply connected subset of $U\cap \overline{U}^{-1}$ containing $\gamma$.)  Our assumptions on $P$ imply that $y\notdivides P$, therefore  $S$ is finite.  Let $V' = V\setminus (\overline{S}^{\,-1})$.  By \lemref{lem:betarootfunction}, $\rho^{\star}(x)$ is a root function for $P^{\star}$ on $V'$.  $T\subset V'$, and for all $\alpha\in T$, $\abs{\rho^{\star}(\alpha)} = 1/\abs{\rho(\alpha)} = 1$ and $\arg \rho^{\star}(\alpha) = \arg\rho(\alpha)$, hence $\rho(\alpha) = \rho^{\star}(\alpha)$.
\begin{enumerate}
	\item  If $T$ is infinite, then then points $\bigl(\alpha,\rho(\alpha)\bigr)$ give infinitely many points in $Z(P)\cap Z(P^{\star})$.  Since $P$ and $P^{\star}$ are both irreducible, this implies that $P\divides P^{\star}$ and $P^{\star}\divides P$ (see \cite[page 2]{shaf}), hence $P$ must be a.s.r.

	\item $P$ being a.s.r.\ implies that $\rho^{\star}(x)$ is actually a root function for $P$ on $V'$ (indeed, on $U$).  By \lemref{lem:algebraic}, distinct root functions of $P$ will never take the same value, but $\rho(\alpha) = \rho^{\star}(\alpha)$, hence $\rho(x)\equiv \rho^{\star}(x)$ on $V'$.  Since $\abs{\rho^{\star}(x)} = 1/\abs{\rho(x)}$ on $\gamma$, we conclude that $\abs{\rho(x)}\equiv 1$ on $\gamma$.  \qedhere
\end{enumerate}
\end{proof}

We will refer to root functions $\rho(x)$ that satisfy $\abs{\rho(x)}\equiv 1$ for all $x$ in some arc $\gamma\subset \UC$ as \emph{toric} root functions, since the map $x\mapsto \bigl(x, \rho(x)\bigr)$, restricted to $\gamma$, parametrizes an infinite subset of $Z(P)\cap \UC^{2}$. The lemma above asserts that $P$ being a.s.r.\ (a ``global'' property of $P$) is a necessary condition for the existence of toric root functions.  It is not a sufficient condition, although experiment suggests that most (perhaps almost all?) a.s.r.\ polynomials in $\C[x,y]$ have at least one toric root function. 



Therefore, the lemma shows that the points in $Z(P)\cap\UC^{2}$ fall into two categories: those lying in the parametrized families coming from toric root functions (which are only possible for $P$ a.s.r.), and finitely many isolated points (which are only possible for $P$ not a.s.r.).

Recall from \thmref{thm:main} that we define $\mathcal{E}$ to be $Z(P)\cap \UC^{2}$ if this is set finite, and $\mathcal{E}=\emptyset$ otherwise.  Therefore, if $P$ is not a.s.r., then $\mathcal{E}=Z(P)\cap \UC^{2}$, and if $P$ is a.s.r., then $\mathcal{E} = \emptyset$ (whether or not $P$ has toric root functions).  This shows that \cororef{cor:zeroseq} follows from \thmref{thm:main}.  It also implies that $\mathcal{E}$ coincides with the set of isolated points of $Z(P)\cap\UC^{2}$.


\begin{lem}\label{lem:algebraiccoords}
If the coefficients of $P$ are algebraic, then the coordinates of the points in $\mathcal{E}$ are algebraic.
\end{lem}

\begin{proof}
If $P$ is a.s.r., then $\mathcal{E} = \emptyset$, so assume otherwise.  Therefore $\mathcal{E} = Z(P)\cap \UC^{2}$, which is $\subseteq Z(P)\cap Z(P^{\star})$, as mentioned before \lemref{lem:betarootfunction}.  Since $P$ is not a.s.r., the latter set is finite.  The $x$ coordinates of all of its points are roots of the resultant $\Res_{y}(P,P^{\star})\in \overline{\Q}[x]$; the $y$ coordinates are handled similarly.
\end{proof}


\begin{lem}\label{lem:rootfnontorus}
Let $f\colon \R\to\R$ be a $C^{\infty}$ function, not identically equal to $nt$ for any integer $n$.  Assume that there exist integers $L, M$, with $L>0$, such that $f(t+2\pi L) - f(t) = 2\pi M$ for all $t\in\R$.  Let $I=[0,2\pi L]$.  Then for all $k\ge 0$ and all integers $n\ge 1$,
\[
\int_{I} \log\abs{1- \exp\bigl(i(f(t)-nt)\bigr)} \,dt = \Osub{k}(n^{-k}).
\]
\end{lem}

\begin{proof}
Our assumptions on $f$ imply that for all $j\ge 1$, $f^{(j)}(t)$ is periodic, hence bounded.  Let $C = \sup_{t\in\R} \abs{f'(t)}$.  
By choosing the implicit constant sufficiently large, the claimed bound is trivially true for $1\le n\le 2C$.  We may therefore restrict to integers $n>2C$.

Recall that the $k$-th polylogarithm function, $\Li{k}(z)$, is continuous on the closed unit disk for $k\ge 1$, with the exception of a singularity at $z=1$ when $k=1$.  (Indeed, $\Li{1}(z) = -\log(1-z)$, using the principal branch of the logarithm for $\abs{z}\le 1$.)  Further, the functions $G_{k}(t) \deq \Li{k}\bigl(\exp(it)\bigr)$ are differentiable on $\R\setminus 2\pi\Z$ for $k\ge 1$, satisfying $G_{k+1}'(t) = i \,G_{k}(t)$.

We define a sequence of functions $\{h_{k}(x,t)\}$ for $k\ge 1$ as follows: let $h_{1}(x,t) = 1$, and for $k\ge 2$, let $h_{k}(x,t) = \frac{\partial}{\partial t}\bigl[h_{k-1}(x,t)/(f'(t)-x)\bigr]$.  By induction, for each $k\ge 1$, we may write $h_{k}(x,t)=P_{k}(x,t)/(x-f'(t))^{2k-2}$, where $P_{k}(x,t)$ is, with respect to $x$, a polynomial of degree $k-1$ whose coefficients are polynomial expressions in $f'(t), f''(t), \dotsc$.    Therefore, substituting $x=n > 2C$, we have $h_{k}(n,t) = \Osub{k}(n^{-(k-1)})$ for $k\ge 1$.

Let $b_{n}\deq \int_{I} \log\abs{1- \exp\bigl(i(f(t)-nt)\bigr)} \,dt$.  (This is finite, as $f(t)\not\equiv nt$.)  Then
\[
b_{n} =  -\Re\biggl[\int_{I} h_{1}(n,t) G_{1}\bigl(f(t)-nt\bigr) \,dt\biggr].
\]
For each $k\ge 1$,
\[
\int_{I} h_{k}(n,t) G_{k}\bigl(f(t)-nt\bigr) dt
	= \int_{I} \frac{-i \,h_{k}(n,t)}{\bigl(f'(t) - n\bigr)} \frac{d}{dt}\Bigl[G_{k+1}\bigl(f(t)-nt\bigr)\Bigr] dt.
\]
Integrating by parts,
\[
= \frac{-i \,h_{k}(n,t)}{\bigl(f'(t) - n\bigr)}G_{k+1}\bigl(f(t)-nt\bigr) \biggr|_{0}^{2\pi L}
	+ i \!\int_{I} h_{k+1}(n,t) G_{k+1}\bigl(f(t)-nt\bigr) dt.
\]
By the properties of $f$, $-i\, h_{k}(n,t) G_{k+1}\bigl(f(t)-nt\bigr) / \bigl(f'(t) - n\bigr)$ takes on the same values at $t=0$ and at $t=2\pi L$.  Therefore,
\[
\int_{I} h_{k}(n,t) G_{k}\bigl(f(t)-nt\bigr) dt =  i \!\int_{I} h_{k+1}(n,t) G_{k+1}\bigl(f(t)-nt\bigr) dt.
\]
Hence, for all $k\ge 1$,
\[
b_{n} =  - \Re\biggl[i^{k-1} \!\int_{I} h_{k}(n,t) G_{k}\bigl(f(t)-nt\bigr) \,dt\biggr] \\
	= \Osub{k}(n^{-(k-1)}).  \qedhere
\]
\end{proof}


\begin{lem}\label{lem:parts}
  Let $\tildeD{}=x \frac{d}{dx}$.
  For $\gamma$ an arc, $k$ a nonnegative integer, $f(x)$ a $C^{k}$ function on a neighborhood of the closure of $\gamma$, and $N$ a nonzero integer,
\begin{equation*}
\begin{split}
    \int_\gamma x^N f(x)\,\frac{dx}{x} &\;=\; -\sum_{r=1}^{k}
    	\frac{1}{(-N)^{r}} \eval{x^{N} \tildeD{r-1}(f)}{\gamma} \\
    & \quad \;\;\;+\,\frac{1}{(-N)^{k}} \int_\gamma x^{N} \tildeD{k}(f)\,\frac{dx}{x}.
\end{split}
\end{equation*}
\end{lem}
The proof is an easy induction, using integration by parts.


\begin{lem}\label{lem:dibrunocor}
Let $k$ be a nonnegative integer and $f$ a $C^{k}$ function on an open subset $U$ of $\C$.  For all integers $m$,
\begin{equation}\label{eqn:diBrunoLem}
\h{m}{k} f(z)
    = f(z)^{m-k}\sum_{j=0}^k \phi_{k,j,f}(z) m^j,
\end{equation}
where $\phi_{k,j,f}(z)\deq\Phi_{k,j}\bigl(f,zf',z^2 f'',\ldots, z^{k-j+1}f^{(k-j+1)}\bigr)$ for $j=0,\dotsc, k$.  (In particular, $\phi_{k,j,f}(z)$ is continuous on $U$ for all $j$.)
\end{lem}

\begin{proof}

Letting $D=\frac{d}{dz}$, an easy induction shows that
\begin{equation}\label{eqn:zDiterated}
(z D)^{k} = \sum_{i}\stirtwosm{k}{i} z^{i} D^{i},
\end{equation}
using the recurrence $\stirtwosm{k}{i} = i \stirtwosm{k-1}{i}+\stirtwosm{k-1}{i-1}$ (valid for all $k$ and $i$ except $k=i=0$) \cite[Equation~1.6.3]{wilf}.

It follows immediately from Fa\`a di Bruno's formula, written in terms of Bell polynomials \cite[Section~3.4, Theorem~A]{comtet1}, that
\begin{equation}\label{eqn:diBruno}
D^{i} \bigl(f(z)^{m}\bigr) = \sum_{l=0}^i (m)_{l} f(z)^{m-l} B_{i,l}(f',f'',\ldots,f^{(i-l+1)}).
\end{equation}
Using the identity $(x)_l = \sum_{j=0}^l (-1)^{l-j} \stironesm{l}{j}x^j$ \cite[Equation~3.5.2]{wilf}, together with \eqref{eqn:zDiterated} and \eqref{eqn:diBruno}, we have
\[
\h{m}{k} f(z) = \sum_{i,l,j}\stirtwosm{k}{i} z^{i} 
	(-1)^{l-j} \stironesm{l}{j} m^j f(z)^{m-l} B_{i,l}(f',f'',\ldots,f^{(i-l+1)}).
\]
By definition, every monomial in $B_{i,l}(y_0,\dotsc,y_{i-l+1})$ has ``weight'' (the sum of the variables' subscripts, with multiplicity) equal to $i$.  Therefore,
\[
	z^{i} B_{i,l}(f',f'',\ldots,f^{(i-l+1)}) = B_{i,l}(z f',z^{2} f'',\ldots,z^{i-l+1} f^{(i-l+1)}).
\]
Together with the previous equation, this proves the formula in the Lemma.

Since $f$ is $C^{k}$, the continuity of the functions $\phi_{k,j,f}(z)$ is clear, except that $\phi_{k,0,f}(z)$ appears to depend on $f^{(k+1)}$.  However, this is illusory, since $\Phi_{k,0}$ is in fact constant, equal to either 0 or 1, for each $k$. \qedhere
\end{proof}


\begin{cor}\label{cor:dibrunoBigO}
  Let $k, m$ be integers with $m\ge 1$ and $k\ge 0$, and let $f$ be a $C^k$ function on an open neighborhood of some compact set $S\subset \C$.  Then for $z\in S$ at which $f(z)\ne 0$,
  \[
  \h{m}{k} f(z) = \Osub{f,S,k}\bigl(m^k \abs{f(z)}^{m-k}\bigr).
  \]
\end{cor}


\begin{lem}\label{lem:realanalytic}
Let $I$ be an interval, $\gamma =$ the image of $I$ under the map $t\mapsto e^{it}$, and $f$ a function holomorphic on $\gamma$.  Let $g(t)\deq |f(e^{it})|$.  Then there exists an open interval $J\supseteq I$ such that $g$ is real analytic on $J\setminus\{t\in\R\,|\,g(t)=0\}$ and continuous on $J$.
\end{lem}

\begin{proof}
Let $V$ be an open neighborhood of $\gamma$ on which $f$ is holomorphic, let $U$ be the preimage of $V$ under the map $t\mapsto e^{it}$, and let $J$ be the connected component of $U\cap\R$ that contains $I$.  Since $g(t)^2 = f(e^{it})\overline{f}(e^{-it})$ on $J$ and $\sqrt{x}$ is real analytic on $(0,\infty)$, the claim follows.
\end{proof}

\begin{lem}\label{lem:integralbounds}
Let $I$ be an interval, $\gamma =$ the image of $I$ under the map $t\mapsto e^{it}$, and $f$ a function holomorphic on the closure of $\gamma$.  Suppose that $|f(z)|\leq 1$ on $\gamma$, with equality at only finitely many points in $\gamma$.  There exists a $\delta>0$ such that for all positive integers $m$ and nonnegative integers $k$,
  \[
  \int_I \bigl|\h{m}{k}(f)(e^{it})\bigr|dt = \Osub{f,I,k}\bigl(m^{k-\delta}\bigr).
  \]
\end{lem}

\begin{proof}
We may restrict our attention to $m\ge 2k$.  Indeed, regardless of what $\delta$ we choose below, by making the implicit constant large enough, we can trivially ensure that the bound is met for $m\in\{1,\dotsc,2k-1\}$.

Let $g(t)=|f(e^{it})|$.  Since $m-k\ge m/2$, \cororef{cor:dibrunoBigO} implies that
\begin{equation*}
\begin{split}
  \int_I \bigl|\h{m}{k}(f)(e^{it})\bigr|dt &= \Osub{f,I,k}\biggl(m^k\int_{I} |f(e^{it})|^{m-k} dt \biggr) \\
  &= \Osub{f,I,k}\biggl(m^k\int_{I}g(t)^{m/2} dt \biggr).
\end{split}
\end{equation*}
It suffices to show that the remaining integral is $\Osub{f,I}(m^{-\delta})$ for some $\delta>0$.

By subdividing $I$ at the (finitely many) values of $t$ where $g(t)=1$ and making a linear change of variables on each subinterval, we may assume without loss of generality that $I$ has the form $[0,b]$ for some $b>0$, and that $g(t)<1$ on $(0,b]$.

If $g(0)=1$, then by \lemref{lem:realanalytic}, there exists $r>0$ such that $g$ is real analytic on $(-r,r)$.  We may take $r<b$.  (If $g(0)<1$, let $r=0$.)

First consider the interval $[r,b]$.  Since $g(t)$ is continuous and $<1$ there, there exists a constant $\lambda<1$ such that $0\le g(t) \le \lambda$ on $[r,b]$.  Hence
\[
\int_{r}^b g(t)^{m/2}dt \le b \lambda^{m/2},
\]
which is $\o(m^{-\delta})$ for every $\delta>0$.

If $g(0)\ne 1$, we are finished.
Otherwise, it remains to bound the integral on $[0,r]$.  Let $\nu$ equal the order of vanishing of $g(t)-1$ at $0$.\footnote{If $f(z)$ is a rational function, it can be shown that $\nu$ is at most twice the maximum of the degrees of the numerator and denominator of $f$.}  There exists a $c>0$ such that $g(t) \le 1-ct^\nu \le \exp(-c t^{\nu})$ for all $t\in [0,r]$.  Therefore,
\[
\int_{0}^{r} g(t)^{m/2}dt 
	\le \int_{0}^{\infty} \! \exp(-cmt^{\nu}/2) \,dt
	= (cm/2)^{-\frac{1}{\nu}} \Gamma\bigl(1+\tfrac{1}{\nu}\bigr)
	= \Osub{f,I}(m^{-\frac{1}{\nu}}).
\]
Letting $\delta = 1/\nu$, we obtain the desired bound.
\end{proof}


Later on, when we apply \lemref{lem:dibrunocor}, the function $f$ will be one of the root functions $\rho(x)$ of $P$.  As it can be difficult (or impossible) to obtain explicit expressions for such algebraic functions, let alone their higher derivatives, the following generalization of implicit differentiation (see \cite{wilde}, \cite{comtet1}, \cite{comtet2}, \cite{comtetfiolet}) is useful.  Recall the polynomials $Q_{n}$ defined in section~\ref{sec:notation}.

\begin{lem}\label{lem:implicitderivs}
Let $\rho(x)$ be a root function for $P(x,y)$ on some open subset of $\C$.  For all integers $n\ge 1$,
\[
\rho^{(n)}(x) = Q_{n}/w_{0,1}^{{2n-1}},
\]
after making the substitutions
	$w_{i,j} = \frac{\partial^{i+j}}{\partial x^{i} \partial y^{j}} P(x,y)$ for all $i,j$ and $y=\rho(x)$.
\end{lem}
For example, when $n=1$, this gives the familiar fact that $dy/dx = -P_{x}/P_{y}$.

Recall the definition of $\Omega_{P,r,a}(x,y)$ from \thmref{thm:main}.
\begin{lem}\label{lem:Psiproperties}
Let $\rho(x)$ be a root function for $P(x,y)$ on some open subset of $\C$.  For integers $2\le a\le r$,
\[
\Omega_{P,r,a}\bigl(x, \rho(x)\bigr) = \phi_{r-1, r-a+1, \rho}(x)/\rho(x)^{r-1}.
\]
\end{lem}

\begin{proof}
Let $\mathbf{v} = \{v_{i,j}\}$ be a doubly-indexed sequence of indeterminates, like the indeterminates $\mathbf{w} = \{w_{i,j}\}$ used by $Q_{n}$ and $\Psi_{r,a}$.  By the restrictions on the monomials in the definition of $Q_{n}$, if we substitute $w_{i,j} = x^{i}y^{j-1} v_{i,j}/v_{0,1}$ for all $i,j$, then $Q_{n}(\mathbf{w})= x^{n} Q_{n}(\mathbf{v})/(y \,v_{0,1}^{2n-1})$.  Therefore, if we substitute $v_{i,j} = \frac{\partial^{i+j}}{\partial x^{i} \partial y^{j}} P$ 
and $y=\rho(x)$, \lemref{lem:implicitderivs} implies that $y\, Q_{n}(\mathbf{w}) = x^{n} \rho^{(n)}(x)$.

By definition, $\Omega_{P,r,a}(x,y)
 = \Psi_{r,a}(\mathbf{w}) 
 = \Phi_{r-1,r-a+1}\bigl(1, Q_{1}(\mathbf{w}), Q_{2}(\mathbf{w}), \dotsc\bigr)$.  Since $\Phi_{r-1,r-a+1}$ is homogeneous of degree $r-1$,
\[
\begin{split}
\Omega_{P,r,a}(x,y) =\;& \Phi_{r-1,r-a+1}\bigl(y, y\,Q_{1}(\mathbf{w}), y\,Q_{2}(\mathbf{w}), \dotsc\bigr)/y^{r-1} \\
=\; & \Phi_{r-1,r-a+1}\bigl(\rho(x), x \rho'(x), x^{2} \rho''(x),\dotsc\bigr)/y^{r-1} \\
=\; & \phi_{r-1, r-a+1,\rho}(x)/\rho(x)^{r-1}. \qedhere
\end{split}
\]
\end{proof}


\section{Proof of \propref{prop:uniqueness}}\label{sec:proofofuniqueness}

\begin{lem}
Let $\omega\colon\R\to \R$ be a quasiperiodic function, and let $\tilde{\omega} = \omega\bigr|_{\N}$, the restriction of $\omega$ to $\N$.  If $\tilde{\omega}(n)$ approaches some (finite) limit as $n\to\infty$, then $\tilde{\omega}$ must be constant.
\end{lem}

\begin{proof}
Assume that $\tilde{\omega}$ is nonconstant.  Choose continuous, periodic functions $\tau_1(t),\dotsc, \tau_k(t)$ such that $\omega(t)=\sum_{i=1}^k \tau_i(t)$.  Without loss of generality, each of the summands $\tau_{i}(t)$ is nonconstant.  For each $i$, let $\alpha_i$ be the period of $\tau_i$.

For real numbers $x$, let $\norm{x}$ denote the distance from $x$ to the nearest integer.  By a result in simultaneous Diophantine approximation \cite[Chapter~I, Theorem~VI]{cassels}, there is an infinite sequence of distinct positive integers $m_j$ such that $\max_{1\le i\le k} \norm{m_j/\alpha_i}\to 0$ as $j\to\infty$.  Therefore, $m_j\to 0$ simultaneously in each of the metric spaces $\R/\alpha_i\Z$ (the natural domain of $\tau_i$).  We conclude that for every positive integer $a$,
\[
    \omega(a) =\sum_{i=1}^k \tau_i(a)= \sum_{i=1}^k \lim_{j\to\infty} \tau_i(a+m_j) = \lim_{j\to\infty} \omega(a+m_j) = \lim_{n\to\infty}\omega(n),
\]
therefore $\tilde{\omega}$ is constant.
\end{proof}

\begin{proof}[Proof of \propref{prop:uniqueness}]
It suffices to prove uniqueness for expansions of $0$.  That is, if a sequence of quasiperiodic functions $\{c_j(t)\}$ satisfies $0 \quasi \sum_{j=0}^{\infty} c_j(n)/n^j$, we wish to show that $c_{j}(n)=0$ for all integers $n\ge 1$ and $j\ge 0$.

We will induct.  Let $k$ be a nonnegative integer, and suppose that $c_{j}(n)\equiv 0$ for $0\le j < k$ and $n\in\N$.  (Conveniently, this also applies for our base case, $k=0$, by vacuity.)  Then
\[
\frac{c_{k}(n)}{n^{k}} = \sum_{j=0}^{k} \frac{c_{j}(n)}{n^{j}} = \Osub{k}\biggl(\frac{1}{n^{k+1}}\biggr),
\]
hence $c_{k}(n)\to 0$ as $n\to\infty$.  Therefore, by the lemma above, $c_{k}(n)=0$ for all positive integers $n$.
\end{proof}


\section{Proof of \thmref{thm:main}}\label{sec:proofofmain}

Since $m(P_1 P_2)=m(P_1)+m(P_2)$, we have $\Delta_{n}(P_{1}P_{2}) = \Delta_{n}(P_{1}) + \Delta_{n}(P_{2})$ for all $n$.  We may therefore assume without loss of generality that $P$ is absolutely irreducible, at least for the purpose of showing existence of an a.p.p.s.\ expansion.

A change of variables shows that $m\bigl(P(x^{M},y^{N})\bigr) = m\bigl(P(x,y)\bigr)$ for all nonzero integers $M, N$.  It follows that if $\deg_{x}(P)$ or $\deg_{y}(P)=0$, then $\mathcal{E} = \emptyset$ and $\Delta_{n}(P)=0$ for all $n$, hence \thmref{thm:main} is trivially true.  We will therefore assume that $\deg_x(P)$ and $\deg_y(P)$ are both $>0$.

Also, if there exists an $k\in\N$ for which $P(x,x^{k})$ is identically zero, then $Z(P)\cap Z(y-x^{k})$ would be infinite, hence $P = c(y-x^{k})$ for some constant $c$, and $\mathcal{E} = \emptyset$.  But $m(y-x^{k}) = 0$, as follows from \eqref{eqn:PxyMM}, and also $m(x^{n}-x^{k})=0$ for all $n\in\N$ by \eqref{eq:onevar}.  Hence, we would have $\Delta_{n}(P) = 0$ for all $n$, and the theorem is again trivially true.  We may therefore assume that, for every positive integer $k$, $P(x,x^{k})$ is not identically zero.

Since $\mathcal{R}(x)$ has no roots on $\UC$, we may choose real numbers  $0<r_{1} < 1 <r_{2}$ such that the annulus $r_{1} < \abs{x} < r_{2}$ contains no roots of $\mathcal{R}(x)$.  Let $\mathcal{E}_{x}$ be the set of all $x$ coordinates of the points in $\mathcal{E}$, and fix an arbitrary point $\slicepoint\in \UC\setminus \mathcal{E}_{x}$.  Define
\[
	W \deq \{x\in \C \,:\, r_{1} < \abs{x} < r_{2} \text{ and } \arg(x) \ne \arg(\slicepoint)\},
\]
which is the annulus mentioned above sliced along the ray from $0$ through $\slicepoint$.  The set $W$ meets the hypotheses of \lemref{lem:algebraic}.
Let
\[
\rtfnset \deq \{\rho_{1}(x),\dotsc, \rho_{d}(x)\}
\]
be the set of all root functions of $P$ on $W$.  We partition $\rtfnset$ into $\rtfnsettor$, the set of toric root functions on $W$ and $\rtfnsetnon = \rtfnset\setminus\rtfnsettor$, the non-toric root functions.

(The branch cut in the definition of $W$ ensures that the functions in $\rtfnset$ are single-valued, but the break it creates in their domains is unnatural.  Each $\rho(x)\in \rtfnset$ may be analytically continued past that ray, becoming a (usually different) element of $\rtfnset$.)

We can now define the function $s:\mathcal{E}\to\{-1,0,1\}$ mentioned in \thmref{thm:main}.  For each point $(\alpha,\beta)\in\mathcal{E}$, there exists a unique, non-toric root function $\rho(x)\in \rtfnsetnon$ such that $\beta = \rho(\alpha)$.  We define $s(\alpha,\beta)$ to be $+1$ (respectively, $-1$) if $\abs{\rho(x)}-1$ changes from negative to positive (respectively, positive to negative) as $x$ moves counterclockwise past $\alpha$ in the unit circle.  If $\abs{\rho(x)}-1$ does not change sign there, let $s(\alpha,\beta)=0$.

In practice, we will usually not be able to obtain explicit formulas for the root functions.  In such cases, the signs $s(\alpha,\beta)$ can be rigorously determined by an algebraic method using only on the coordinates $(\alpha,\beta)$ and the coefficients of $P$.  This method is described in section~\ref{sec:signdet}.

Subdivide $\UC$ at the points in $\mathcal{E}_{x}\cup\{\slicepoint\}$, obtaining a finite set of open, proper arcs $\Gamma = \{\gamma_{1},\dotsc,\gamma_{M}\}$.  For each arc $\gamma\in \Gamma$ and each non-toric root function $\rho(x)\in \rtfnsetnon$, $\abs{\rho(x)}$ is either always $>1$ or always $<1$ on $\gamma$.  Partition $\Gamma\times \rtfnsetnon$ into the subsets
\[
\begin{split}
\mathcal{A}_{-}	&\deq \bigl\{(\gamma,\rho)\in\Gamma \times \rtfnsetnon :
	 \abs{\rho(x)}>1 \text{ on } \gamma\bigr\}, \\
\mathcal{A}_{+}	&\deq \bigl\{(\gamma,\rho)\in\Gamma \times \rtfnsetnon :
	 \abs{\rho(x)}< 1 \text{ on } \gamma\bigr\}. \\
\end{split}
\]
(The choice of which is $+$ or $-$ is slightly more convenient this way.)

We now extend the domain of $s\colon \mathcal{E}\to \{-1,0,1\}$ by setting $s(\alpha,\beta)=0$ for all $(\alpha,\beta)\notin \mathcal{E}$.  Recall the definition of $\eval{f(x)}{\gamma}$ from section~\ref{sec:notation}.
For a fixed choice of $\pm$, and for any operator $H$ that sends holomorphic functions to continuous functions, careful bookkeeping of signs shows that
\begin{equation}\label{eq:evalsum}
  \sum_{(\gamma,\rho)\in\mathcal{A}_{\pm}} \eval{H\rho(x)}{\gamma} = \pm \sum_{\rho\in \rtfnsetnon}\sum_{\alpha\in \mathcal{E}_{x}} s\bigl(\alpha,\rho(\alpha)\bigr) H \rho(\alpha).
\end{equation}
(For certain root functions $\rho(x)\in\rtfnsetnon$, a copy of $\pm \lim_{x\to \slicepoint} H \rho(x)$ (as $x$ approaches $\slicepoint$ along $\UC$, from one direction or the other) may appear on the left side, but this will be cancelled by the contributions from other root functions.)

With this setup, we are now ready to examine $\Delta_{n}(P)$.

\begin{proof}[Proof of \thmref{thm:main}]

The strategy of the proof is inspired by that used by Boyd in his proof of \propref{prop:boydasymp}.

We first rewrite $m\bigl(P(x,y)\bigr)$, using a standard trick.
\begin{equation*}
m\bigl(P(x,y)\bigr) = \frac{1}{2\pi i} \int_{\abs{x}=1}
	\biggl[ \frac{1}{2\pi i} \int_{\abs{y}=1} \log\bigl|P(x,y)\bigr| \frac{dy}{y} \biggr] \frac{dx}{x}
\end{equation*}
The expression inside the brackets is actually a one-variable Mahler measure with respect to $y$, with $x$ viewed there as constant.  Therefore, 
\eqref{eq:onevar} implies that
\begin{align}\label{eqn:PxyMM}
\nonumber
m\bigl(P(x,y)\bigr) &= \frac{1}{2\pi i} \int_{\UC}
		\biggl(\log\bigl|a_{d}(x)\bigr|
		+ \sum_{\rho\in \rtfnset} \logp\bigl|\rho(x)\bigr|\biggr) \frac{dx}{x} \\[2mm]
  &= m(a_{d}) + \frac{1}{2\pi i} \sum_{(\gamma,\rho)\in\mathcal{A}_{-}} \int_{\gamma}
  		\log \bigl|\rho(x)\bigr| \frac{dx}{x}.
\end{align}

We now rewrite $m\bigl(P(x,x^{n})\bigr)$.  Recall that for every positive integer $n$, $P(x,x^{n})$ is not identically $0$.
\begin{equation*}
\begin{split}
m\bigl(P(x,x^{n})\bigr) &= \frac{1}{2\pi i} \int_{\UC} \log\bigl|P(x,x^{n})\bigr| \frac{dx}{x} \\[2mm]
  &= \frac{1}{2\pi i} \int_{\UC} \biggl(\log\bigl|a_{d}(x)\bigr|
		+ \sum_{\rho\in \rtfnset} \log\bigl|x^{n}-\rho(x)\bigr|\biggr) \frac{dx}{x} \\[2mm]
  &= m(a_{d}) + \frac{1}{2\pi i} \int_{\UC}
  		\sum_{\rho\in \rtfnset} \log\bigl|x^{n}-\rho(x)\bigr|\frac{dx}{x}.  \\[2mm]
\end{split}
\end{equation*}
Since $P(x,x^{n})\not\equiv 0$, there are, for each $\rho(x)\in \rtfnset$, only finitely many $x\in W$ such that $x^{n} = \rho(x)$.
It follows that $\log\bigl|x^{n}-\rho(x)\bigr|$ is integrable on $\UC$.  Therefore, we may interchange the sum and integral above, obtaining
\begin{equation}\label{eqn:PxxnMM}
m\bigl(P(x,x^{n})\bigr) = m(a_{d}) + \frac{1}{2\pi i} \sum_{\rho\in\rtfnset}
	\int_{\UC} \log\bigl|x^{n}-\rho(x)\bigr|\frac{dx}{x}.
\end{equation}

We now separately consider those terms in \eqref{eqn:PxxnMM} coming from toric root functions.  Given one such $\rho(x)\in\rtfnsettor$ and some $t'\in\R$ such that $\exp(i t')\ne \slicepoint$, there exists a holomorphic function $f(z)$, defined on a neighborhood of $t'$ in $\C$, such that $f(t)\in\R$ for $t\in\R$ and $\rho\bigl(\exp(i z)\bigr) = \exp\bigl(i f(z)\bigr)$.  We analytically continue $f(z)$ to a strip $-\delta<\Im(z)<\delta$ for some $\delta>0$.  If we let $z$ move to the right in steps of $2\pi$, $\exp\bigl(i f(z)\bigr)$ will cycle through values of $\tilde{\rho}\bigl(\exp(i z)\bigr)$ for various $\tilde{\rho}(x)$ in some subset of $\rtfnsettor$.  Letting $L$ equal the size of this orbit, $\exp\bigl(i f(z+2\pi L)\bigr) \equiv \exp\bigl(i f(z)\bigr)$.  Therefore, there exists an integer $M$ such that $f(z+2\pi L) = f(z) + 2\pi M$ for all $z$ in the strip.  Also, $f(t)\not\equiv nt$ for every integer $n$, as otherwise $\rho(x)\equiv x^{n}$, hence $P(x,x^{n})\equiv 0$.

Each function $f(z)$ described above encapsulates some subset of $\rtfnsettor$.  Therefore, the terms in \eqref{eqn:PxxnMM} coming from the toric root functions can be written as the sum of finitely many expressions of the form
\[
\frac{1}{2\pi} \int_{0}^{2\pi L} \log\abs{1-\exp\bigl(i (f(t)-nt)\bigr)} dt,
\]
for some $f$ as described above.  Invoking \lemref{lem:rootfnontorus}, this integral is $\Osub{k}(n^{-k})$ for all $k\ge 0$.
Therefore, for all $k\ge 0$,
\[
m\bigl(P(x,x^{n})\bigr) = m(a_{d}) + \frac{1}{2\pi i} \sum_{\rho\in\rtfnsetnon}
	\int_{\UC} \log\bigl|x^{n}-\rho(x)\bigr|\frac{dx}{x} + \Osub{k}(n^{-k}).
\]
Split $\UC$ into the arcs $\gamma\in\Gamma$, and rewrite $\log\bigl|x^{n}-\rho(x)\bigr|$ as $\log\bigl|1- \rho(x)/x^{n}\bigr|$ if $(\gamma,\rho)\in \mathcal{A}_{+}$, and as $\log\bigl|\rho(x)\bigr|+ \log\bigl|1-x^{n}/\rho(x)\bigr|$ if $(\gamma,\rho)\in \mathcal{A}_{-}$.  Combining the equation above with \eqref{eqn:PxyMM} and using the shorthand $\sum_{\pm} c_{\pm}$ for $c_{+} + c_{-}$,
\begin{align}\label{eqn:Dnptemp}
  \nonumber
  \Delta_{n}(P) &= \frac{1}{2\pi i} \sum_{\pm} \sum_{(\gamma,\rho)\in\mathcal{A}_{\pm}} \int_{\gamma}
		\log\Bigl|1- \Bigl(\frac{\rho(x)}{x^{n}}\Bigr)^{\pm 1}\Bigr| \frac{dx}{x}
		+ \Osub{k}(n^{-k}) \\[2mm] \nonumber
	&= \frac{1}{2\pi} \sum_{\pm} \sum_{(\gamma,\rho)\in\mathcal{A}_{\pm}} \int_{\gamma} 
		-\Re\biggl(\,\sum_{m=1}^{\infty}
		 \frac{1}{m}\Bigl(\frac{\rho(x)}{x^{n}}\Bigr)^{\pm m} \biggr) d\arg(x) + \Osub{k}(n^{-k})
		  \\[2mm]
	&= -\frac{1}{2\pi}  \Im\biggl(\sum_{\pm} \sum_{(\gamma,\rho)\in\mathcal{A}_{\pm}} \int_{\gamma} \sum_{m=1}^{\infty}
		 \frac{1}{m}\Bigl(\frac{\rho(x)}{x^{n}}\Bigr)^{\pm m} \frac{dx}{x} \biggr) + \Osub{k}(n^{-k}).
\end{align}

By Beppo Levi's lemma \cite[Theorem~10.26]{apostol1}, together with the $k=0$ case of \lemref{lem:integralbounds}, we may interchange the integral and the sum over $m$ in \eqref{eqn:Dnptemp}.
Therefore, if we define
\[
  \ipm \deq \sum_{(\gamma,\rho)\in\mathcal{A}_{\pm}} \int_{\gamma} \Bigl(\frac{\rho(x)}{x^{n}}\Bigr)^{\pm m} \frac{dx}{x},
\]
then for all $k\ge 0$,
\begin{equation}\label{eqn:Dnpnice}
  \Delta_{n}(P) = -\frac{1}{2\pi} \sum_{m=1}^{\infty} \frac{1}{m}
    \sum_{\pm} \Im\bigl(\ipm\bigr) + \Osub{k}(n^{-k}).
\end{equation}

For integers $n,m, r\ge 1$ and $k\ge 0$, let
\[
\Spm \deq \pm \!\!\sum_{\rho\in \rtfnsetnon} \sum_{\alpha\in \mathcal{E}_{x}} s\bigl(\alpha,\rho(\alpha)\bigr)\, \alpha^{\mp nm} \h{\pm m}{r-1} \rho(\alpha)
\]
and
\[
\jpmk \deq \!\!\!\sum_{(\gamma,\rho)\in \mathcal{A}_{\pm}} \int_{\gamma} x^{\mp nm} \h{\pm m}{k} \rho(x)  \,\frac{dx}{x}.
\]
\lemref{lem:parts}, together with \eqref{eq:evalsum}, then implies that
\[
  \ipm = -\sum_{r=1}^{k} \frac{1}{(\pm nm)^{r}} \Spm + \frac{1}{(\pm nm)^{k}} \jpmk.
\]
Applying \lemref{lem:integralbounds} to $\jpmk$ (with $\h{\pm m}{k} \rho(x)$ rewritten as $\h{m}{k} \bigl[\rho(x)^{\pm 1}\bigr]$), there exists a $\delta>0$ such that for $n, m\ge 1$ and $k\ge 0$, $\jpmk = \Ok(m^{k-\delta})$.  Hence
\[
    \ipm =
    -\sum_{r=1}^{k} \frac{\Spm}{(\pm nm)^{r}}
    + \Ok\bigl(n^{-k} m^{-\delta}\bigr).
\]

By \cororef{cor:dibrunoBigO}, $S^{\pm}_{n,m,k} = \Ok(m^{k-1})$, hence the $r=k$ term in the sum above may be absorbed into the error term.  It follows that
\[
\sum_{m=1}^\infty \frac{1}{m}\ipm 
	= -\sum_{m=1}^\infty \frac{1}{m}\sum_{r=1}^{k-1} \frac{\Spm}{(\pm nm)^r}
	\,+\, \Ok(n^{-k}).
\]
(Since $\ipm = \jpmzero = \Osub{}(m^{-\delta})$, the sums above converge.)

Therefore, by \eqref{eqn:Dnpnice}, for all $n\ge 1$ and $k\ge 0$,
\begin{equation}\label{eq:almost}
  \Delta_{n}(P) = \sum_{r=1}^{k-1} \frac{c_r(n)}{n^r}  + \Ok(n^{-k}),
\end{equation}
where
\[
  c_r(n) \deq \frac{1}{2\pi}\sum_{m=1}^\infty \frac{1}{m^{r+1}} \sum_\pm (\pm 1)^r\Im\bigl(\Spm\bigr).
\]

We now have a series expansion for $\Delta_{n}(P)$.  The coefficients $c_r(n)$ are independent of $k$, as one would hope.  To finish the proof of \thmref{thm:main}, we still need to rewrite the coefficients in the form given in \eqref{eq:maincoeffformula} and show that they are quasiperiodic.

But first, we will prove that $c_{1}(n) = 0$.
For all $n, m\ge 1$,
\[
\Spmone = \pm\!\!\sum_{\rho\in \rtfnsetnon}\sum_{\alpha\in \mathcal{E}_{x}} s\bigl(\alpha,\rho(\alpha)\bigr)\, \Bigl(\frac{\rho(\alpha)}{\alpha^{n}}\Bigr)^{\pm m}
	= \pm \!\!\!\sum_{(\alpha,\beta)\in\mathcal{E}} s(\alpha,\beta) (\beta/\alpha^{n})^{\pm m}.
\]
For $(\alpha,\beta)\in \mathcal{E}$, $(\beta/\alpha^{n})^{-m}=\overline{(\beta/\alpha^{n})^m}$.  Therefore,
\[
\sum_\pm \pm \Spmone =
	2 \!\!\sum_{(\alpha,\beta)\in \mathcal{E}} s(\alpha,\beta) \Re\bigl((\beta/\alpha^{n})^{m}\bigr).
\]
Since this expression is real, $\sum_\pm \pm \Im\bigl(\Spm\bigr) = \Im\bigl(\sum_\pm \pm \Spm\bigr) =0.$  It follows that $c_{1}(n)=0$.

Consequently, we can restrict to $r\ge 2$ in \eqref{eq:almost}.

Recall the functions $\phi_{k,j,f}(x)$ described in \lemref{lem:dibrunocor}.  By that lemma,
\[
\Spm = \pm\sum_{\rho\in \rtfnsetnon} \sum_{\alpha\in \mathcal{E}_{x}} s\bigl(\alpha,\rho(\alpha)\bigr)
		\Bigl(\frac{\rho(\alpha)}{\alpha^{n}}\Bigr)^{\!\pm m} \,
     		\sum_{j=0}^{r-1} \frac{\phi_{r-1,j,\rho}(\alpha)}{\rho(\alpha)^{r-1}} (\pm m)^j.
\]
Since $r\ge 2$, $\phi_{r-1,0,\rho}=0$.  Therefore, we may restrict to $j\ge 1$.
By \lemref{lem:Psiproperties},
\[
\Spm = \pm\!\!\!\sum_{(\alpha,\beta)\in \mathcal{E}} s(\alpha,\beta) (\beta/\alpha^{n})^{\pm m}
	\sum_{j=1}^{r-1} \Omega_{P,r,r-j+1}(\alpha,\beta) (\pm m)^{j}.
\]
Therefore,
\[
\begin{split}
  c_r(n) &=  \frac{1}{2\pi} \sum_{m=1}^\infty \frac{1}{m^{r+1}} \sum_\pm (\pm 1)^{r} \Im\bigl(\Spm\bigr) \\
  &= \frac{1}{2\pi} \sum_{(\alpha,\beta)\in \mathcal{E}} s(\alpha,\beta) \sum_{j=1}^{r-1}
  		\Im\biggl(\Omega_{P,r,r-j+1}(\alpha,\beta) \sum_\pm
            \sum_{m=1}^\infty \frac{(\beta/\alpha^{n})^{\pm m}}{(\pm m)^{r-j+1}} \biggr).
\end{split}
\]
Letting $a=r-j+1$,
\[
  c_r(n) = \frac{1}{2\pi} \!\!\sum_{(\alpha,\beta)\in \mathcal{E}} \!\! s(\alpha,\beta) \sum_{a=2}^{r}
  		\Im\Bigl(\Omega_{P,r,a}(\alpha,\beta) \sum_\pm (\pm 1)^{a}
            \Li{a}\bigl((\beta/\alpha^{n})^{\pm 1}\bigr)\Bigr).
\]
Recall from \thmref{thm:main} that we define $\RR{a}(z)$ to be $\Re(z)$ for $a$ even and $\Im(z)$ for $a$ odd.
Like before, $(\alpha,\beta)\in\mathcal{E}$ implies
$\Li{a}\bigl((\beta/\alpha^{n})^{-1}\bigr) = \Li{a}\bigl(\overline{\beta/\alpha^{n}}\bigr) = \overline{\Li{a}(\beta/\alpha^{n})}$.  
Therefore, for each $(\alpha,\beta)\in\mathcal{E}$,
\[
\begin{split}
\Im\biggl(\Omega_{P,r,a}&(\alpha,\beta) \sum_\pm (\pm 1)^{a}
		\Li{a}\bigl((\beta/\alpha^{n})^{\pm 1}\bigr) \biggr) \\
    =&\, \Im\biggl(\Omega_{P,r,a}(\alpha,\beta)
    		\,2\RR{a}\bigl(\Li{a}(\beta/\alpha^{n}) \bigr)
		\begin{cases}
		1 & \text{$a$ even} \\[-1.5mm]
		i & \text{$a$ odd}
		\end{cases}
		\biggr) \\[1mm]
    =&\,  2\RR{a+1}\bigl(\Omega_{P,r,a}(\alpha,\beta)\bigr) \RR{a}\bigl(\Li{a}(\beta/\alpha^{n})\bigr),
\end{split}
\]
which proves \eqref{eq:maincoeffformula}.

For any given point $(\alpha,\beta)\in \mathcal{E}$, let $\theta$ be some value for $\arg(\alpha)$, so we can write $\beta/\alpha^{n} = \beta \exp(-i\theta n)$.  Replacing $n$ with a real parameter $t$, this gives a  periodic function of $t$.  Since $a\ge 2$, $\Li{a}(z)$ is continuous on $\UC$.
Hence $\RR{a}\bigl(\Li{a}(\beta \exp(-i \theta t))\bigr)$ is a continuous, periodic function of $t$.  (And if $\alpha$ is a $k$-th root of unity, the function has period dividing $k$.)
It follows that $c_{r}(n)$ is the restriction to $\N$ of a quasiperiodic function.  This concludes the proof of \thmref{thm:main}.
\end{proof}


\section{Proof of \cororef{cor:main}}\label{sec:proofofcor}

We recall here the formula given in \cororef{cor:main}.  
Let $\xi = \exp(2\pi i/3)$.  Then $\Delta_{n}(1+x+y) \quasi \sum_{r=2}^{\infty} c_{r}(n)/n^{r}$, for
\begin{equation}\label{eq:cororeminder}
c_r(n) =  \frac{2}{\pi}
	 \sum_{\substack{a,b \\ a+b=r+1}} 
	 	\!\!\!\RR{a}\bigl(\Li{a}(\xi^{n+1})\bigr) \sum_{j=b}^{r-1}
	 (-1)^{b} \stirone{j}{b}\stirtwo{r-1}{j} \RR{a+1}(\xi^{j}).
\end{equation}

\begin{proof}[Proof of \cororef{cor:main}]
For this $P$, we have a unique root function: $\rho(x)= -x-1$.  Any point $(\alpha,\beta)\in Z(P)\cap \UC^{2}$ must satisfy $\abs{\alpha}=\abs{\alpha+1}=1$.  It follows that $\mathcal{E} =\bigl\{(\xi, \xi^{-1}), (\xi^{-1},\xi)\bigr\}$.  The associated signs are $s(\xi^{\mp}, \xi^{\pm}) = \pm 1$.

$P$ satisfies the hypotheses of \thmref{thm:main}.  Since our root function is so simple, it is easier to calculate $\Omega_{P,r,a}(x,y)$ with \lemref{lem:Psiproperties}, bypassing $\Psi_{r,a}$:
\[
\begin{split}
\Omega_{P,r,a}(x,y) &= \phi_{r-1,r-a+1}(x)/y^{r-1} \\
	&=\Phi_{r-1,r-a+1}\bigl(\rho(x), x\rho'(x),x^{2}\rho''(x),\dotsc\bigr)/y^{r-1} \\
	&= \Phi_{r-1,r-a+1}\bigl(-x-1, -x,0,0,\dotsc\bigr)/y^{r-1}.
\end{split}
\]
To calculate this, we need to evaluate the Bell polynomials $B_{j,i}(-x,0,\dotsc,0)$ for various $i,j$.  It follows from the definitions that $B_{j,i}(-x,0,\dotsc,0) = 0$ unless $i=j$, and that $B_{j,j}(-x,0,\dotsc,0) = (-x)^{j}$.  Therefore, letting $b=r-a+1$ for brevity,
\[
\Omega_{P,r,a}(x,y) = \frac{1}{y^{r-1}} \sum_{j} (-1)^{j-b} 
	\stirone{j}{b} \stirtwo{r-1}{j} (-x-1)^{r-j-1}(-x)^{j}.
\]
If we substitute a point $(\alpha,\beta)\in\mathcal{E} = \{(\xi^{\mp}, \xi^{\pm})\}$, then since $\beta= -\alpha-1 = 1/\alpha = \alpha^{2}$ for these points,
\[
\Omega_{P,r,a}(\alpha,\beta) = (-1)^{b} \sum_{j} \stirone{j}{b} \stirtwo{r-1}{j} \alpha^{-j}.
\]
We now apply \eqref{eq:maincoeffformula}.  Instead of summing over $(\alpha,\beta)\in\mathcal{E}$, we substitute $(\alpha,\beta) = (\xi^{-\sigma},\xi^{\sigma})$ and sum over $\sigma=\pm 1$.  (This way, $s(\xi^{-\sigma},\xi^{\sigma}) = \sigma$.)
\[
\begin{split}
  c_r(n)
    &= \frac{1}{\pi}\sum_{\sigma=\pm 1} \sigma \sum_{a=2}^{r} 
    		\RR{a+1}\bigl(\Omega_{P,r,a}(\xi^{-\sigma},\xi^{\sigma})\bigr) \RR{a}\bigl(\Li{a}(\xi^{\sigma(n+1)})\bigr) \\
    &= \frac{1}{\pi}\sum_{\sigma=\pm 1} \,\sigma \!\!\!\sum_{\substack{a+b = r+1,\\ 2\le a \le r}} \!\!
    	 \RR{a}\bigl(\Li{a}(\xi^{\sigma(n+1)})\bigr)
    	(-1)^{b} \sum_{j} \stirone{j}{b} \stirtwo{r-1}{j} \RR{a+1}\bigl(\xi^{\sigma j}\bigr)
\end{split}
\]

The restriction that $2\le a \le r$ may be dropped, since the summands vanish if $a< 1$ or $a> r$.

When $\sigma=-1$, the arguments to both $\RR{a}$ and $\RR{a+1}$ above become the conjugates of the values they take when $\sigma=+1$.  A careful examination of signs then shows that the $\sigma=-1$ term is identical to the $\sigma=+1$ term.  Therefore, we may drop all occurrences of $\sigma$ and double the expression, obtaining \eqref{eq:cororeminder}.
\end{proof}


\section{The coefficients $c_{2}(n)$ and $c_{3}(n)$}\label{sec:earlycoeffs}

The formulas for the coefficients in \thmref{thm:main} and \cororef{cor:main} are so complicated that it is difficult to see what they are actually saying, even in simple cases.  In this section, we will see in more detail what the first two nontrivial coefficients look like, both for general $P$ and for $P=1+x+y$.

For what follows, let $P_{i,j}$ denote $\frac{\partial^{i+j}}{\partial x^{i} \partial y^{j}} P$.
\begin{prop}\label{prop:earlycoeffsgeneral}
Let $P(x,y)$, $\mathcal{E}$, and $s\colon \mathcal{E}\to\{-1,0,1\}$ be as in \thmref{thm:main}.  Then for all $n$,
\[
  c_2(n)
    = \frac{1}{\pi}\sum_{(\alpha,\beta)\in \mathcal{E}} s(\alpha,\beta)
    		\Im(F_{P}) \Re\bigl(\Li{2}(\beta/\alpha^n)\bigr)
\]
and
\[
  c_3(n)
    = \frac{1}{\pi} \!\!\sum_{(\alpha,\beta)\in \mathcal{E}} \!\!s(\alpha,\beta)\Bigl[
    	\Im(F_{P}^{2}) \Re\Bigl(\Li{2}\Bigl(\frac{\beta}{\alpha^{n}}\Bigr)\Bigr)
	+ \Re(G_{P}) \Im\Bigl(\Li{3}\Bigl(\frac{\beta}{\alpha^{n}}\Bigr)\Bigr)
	\Bigr],
\]
where
\[
F_{P}(x,y) = -\frac{x P_{1,0}}{y P_{0,1}}
\]
and
\[
G_{P}(x,y) = \Bigl(-\frac{P_{2,0}}{P_{0,1}} 
 + \frac{2P_{1,0} P_{1,1}}{P_{0,1}^{2}} 
 - \frac{P_{0,2} P_{1,0}^{2}}{P_{0,1}^{3}}\Bigr)\frac{x^{2}}{y}
 + F_{P} - F_{P}^{2}.
\]
\end{prop}

\begin{proof}
The formulas follows directly from \thmref{thm:main}, using the values of $\Psi_{2,2}$, $\Psi_{3,2}$, and $\Psi_{3,3}$ found in Table~\ref{PsiTable}.
\end{proof}

Things get much nicer when one specializes to a simple polynomial, like $1+x+y$, as then all but a few of the partials vanish.

\begin{prop}  Let $\xi = \exp(2\pi i/3)$.  For $P=1+x+y$,

\[
c_2(n) = -\frac{\sqrt{3}}{\pi} \Re\bigl(\Li{2}(\xi^{n+1})\bigr)
\]
and
\[
c_3(n) = \frac{1}{\pi}\Bigl(2\Im\bigl(\Li{3}(\xi^{n+1})\bigr) - \sqrt{3} \Re\bigl(\Li{2}(\xi^{n+1})\bigr)\Bigr).
\]
\end{prop}

Before we prove this, note that $\Li{2}(1) = \zeta(2) = \pi^{2}/6$.  And because
\[
\Li{2}(1) + \Li{2}(\xi) + \Li{2}(\xi^{2}) = \sum_{n=1}^{\infty} \frac{(1^{n}+\xi^{n}+\xi^{2n})}{n^{2}}
= \sum_{m=1}^{\infty} \frac{3}{(3m)^{2}} = \frac{\zeta(2)}{3},
\]
it follows that $\Re\bigl(\Li{2}(\xi^{\pm 1})\bigr) -\pi^{2}/18$.  Therefore, the formula for $c_{2}(n)$ above simplifies to the one given in Boyd's \propref{prop:boydasymp}.

\begin{proof}
As in the proof of \cororef{cor:main}, the elements of $\mathcal{E}$ are $(\alpha,\beta) = (\xi^{-\sigma},\xi^{\sigma})$, with sign $\sigma$, for $\sigma=\pm 1$.  For $P=1+x+y$, using the notation of \propref{prop:earlycoeffsgeneral}, $F_{P} = -\xi^{\sigma}$.  Therefore, $\Im(F_{P}) = \Im(F_{P}^{2}) = -\sigma\sqrt{3}/2$, and $G_{P} = -(\xi^{\sigma} + \xi^{-\sigma}) = 1$.  The claims now follows easily from \propref{prop:earlycoeffsgeneral}.
\end{proof}


\section{Numerical Evidence}\label{sec:numerics}

Let $\sum_{r=2}^{\infty} c_{r}(n)/n^{r}$ be the a.p.p.s.\ expansion for $\Delta_{n}(1+x+y)$.  We have calculated the coefficients $c_{r}(n)$ to high precision for $2\le r\le 200$ and $n\in\{0,1,2\}$ (which covers all $n$, since $c_{r}(n)$ depends only on $n \jmod{3}$).  The values with $r\le 10$ are shown in Table~\ref{BoydCoeffTable}.  We have also calculated $\Delta_{n}(1+x+y)$ up to about $n=300$.


\begin{table}[btp]
  \centering
  \begin{tabular}{c r@{.}l r@{.}l r@{.}l}
     \thickmidrule
      $r$ & 
      	\multicolumn{2}{c}{$c_{r}(1)$} &
      	\multicolumn{2}{c}{$c_{r}(2)$} &
      	\multicolumn{2}{c}{$c_{r}(3)$} \\ \thickmidrule
      	2  &  
 $0$&3022998940  &  
$-0$&9068996821  &  
 $0$&3022998940  \\ \midrule
3  &  
$-0$&1850879776  &  
$-0$&9068996821  &  
 $0$&7896877657  \\ \midrule
4  &  
$-1$&3056972851  &  
 $1$&1934321459  &  
 $0$&1564663299  \\ \midrule

5  &  
 $2$&6830358119  &  
 $3$&2937639738  &  
$-5$&5860975105  \\ \midrule
6  &  
 $24$&5796866908  &
$-26$&3505959270  &
   $1$&4699140266  \\ \midrule
7  &  
 $-79$&4351285864  &  
 $-87$&7396475567  &  
 $165$&1438848791  \\ \midrule
8  &  
 $-953$&9471505649  &  
 $971$&0686842267  &  
 $-15$&3305415600  \\ \midrule
9  &  
 $4176$&6470338729  &  
 $4302$&0164745649  &  
 $-8460$&9088755602  \\ \midrule
10  &  
 $62932$&4226515189 &  
 $-63281$&1309582098  &  
 $331$&7360326211  \\ \thickmidrule
  \end{tabular}

  \caption{The coefficients $c_{r}(n)$ in the expansion for $\Delta_{n}(1+x+y)$.}
  \label{BoydCoeffTable}
\end{table}


\subsection{Confirmation of \cororef{cor:main}}

Let $p_{k}(n)\deq \sum_{r=2}^{k} c_{r}(n)/n^{r}$, the $k$-th partial sum of the expansion for $\Delta_{n}(1+x+y)$.  By the definition of the expansion,
\[
\Delta_{n}(1+x+y) = p_{k-1}(n) + \frac{c_{k}(n)}{n^{k}} + \Ok\Bigl(\frac{1}{n^{k+1}}\Bigr).
\]
Therefore, for any fixed $n_{0}\in\{0,1,2\}$,
\begin{equation}\label{eq:approxformula}
c_{k}(n_{0}) = \lim_{\substack{\\[-1 pt] n\to\infty, \\ n\equiv n_{0} \jmod{3}}}
	n^{k}\bigl(\Delta_{n}(1+x+y) - p_{k-1}(n)\bigr).
\end{equation}
With this formula, one may numerically verify each of the coefficients $c_{k}(n_{0})$ obtained from \cororef{cor:main}, assuming the accuracy of the coefficients with smaller $k$.  Table~\ref{BoydCoeffVerification} demonstrates this verification for $n_{0}=1$ and $k=2, 3, 4$.  Using $n\approx 300$ and $2\le k\le 30$, we found good agreement between the left and right sides of \eqref{eq:approxformula} (with the relative error frequently $< 0.05$ and always $< 0.17$), with the exception of $n_{0}=0$ and $k$ even, when convergence was much slower.  Using Richardson extrapolation dramatically improved convergence in all cases.  Overall, the data gives strong evidence for the validity of \cororef{cor:main}.  Numerical tests with other polynomials were also favorable.


\begin{table}[btp]
  \centering
  \begin{tabular}{l r@{.}l r@{.}l r@{.}l}
     \thickmidrule
      $n$ & 
      	\multicolumn{2}{c}{$n^{2} \Delta_{n}$} &
      	\multicolumn{2}{c}{$n^{3}\bigl(\Delta_{n} - \frac{c_{2}(n)}{n^{2}}\bigr)$} &
      	\multicolumn{2}{c}{$n^{4}\bigl(\Delta_{n} - \frac{c_{2}(n)}{n^{2}} 
			- \frac{c_{3}(n)}{n^{3}}\bigr)$} \\ \thickmidrule
1  &   $0$&3700812333  &   $0$&0677813393  &  
$0$&2528693169  \\ \midrule
61  &   $0$&2989282502  &  $-0$&2056702769  & 
$-1$&2555202582  \\ \midrule
121  &   $0$&3006826863  &  $-0$&1956821406  & 
$-1$&2818937239  \\ \midrule
181  &   $0$&3012379282  &  $-0$&1922158113  & 
$-1$&2901378960  \\ \midrule
241  &   $0$&3015096124  &  $-0$&1904578830  & 
$-1$&2941471883  \\ \midrule
301  &   $0$&3016706736  &  $-0$&1893953380  & 
$-1$&2965154617  \\ \thickmidrule
  \end{tabular}

  \caption{The scaled difference between $\Delta_{n} \deq \Delta_{n}(1+x+y)$ and the first few partial sums of its a.p.p.s.\ expansion, for selected values of $n\equiv 1 \jmod{3}$.
  The numbers in columns 2--4 approach the values of $c_{2}(1)$, $c_{3}(1)$, and $c_{4}(1)$ found in Table~\ref{BoydCoeffTable}.}
  \label{BoydCoeffVerification}
\end{table}


\subsection{Approximation of $\Delta_{n}(1+x+y)$ with partial sums}

Continuing to use $P=1+x+y$,
numerics strongly suggest that for each $n$, the power series $\sum_{r=2}^{\infty} c_{r}(n) x^{r}$ has radius of convergence 0, hence $\sum_{r=2}^{\infty} c_{r}(n)/n^{r}$ is divergent.  However, as is common with asymptotic expansions, partial sums give good approximations for $\Delta_{n}(1+x+y)$ up to a point, before going wildly off the mark.  Let $p_{k}(n)\deq \sum_{r=2}^{k} c_{r}(n)/n^{r}$, the $k$-th partial sum, and let $K(n)$ be the value of $k$ for which $p_{k}(n)$ gives the best approximation of $\Delta_{n}(1+x+y)$.  The numerical evidence suggests that as $n\to\infty$, $K(n)$ is asymptotic to $n$, and that the relative error $\bigl(p_{n}(n)-\Delta_{n}\bigr)/\Delta_{n} = \Osub{}\bigl(\exp(-n)\bigr)$.
For example, for $n=100$, the best approximation of $\Delta_{100}(1+x+y)$ is given by $p_{107}(100)$, which is correct to 46 decimal places---a relative error of $1.5\times 10^{-43}$.


\subsection{A polynomial not meeting the hypothesis}\label{subsec:deninger}

Our proof of \thmref{thm:main} does not make it clear what to expect for polynomials $P$ for which $P$ and $\Pyinline$ do have a common zero on $\UC^{2}$.  We chose such a polynomial---namely, $P = 1+ x + 1/x + y + 1/y$ (famous for Deninger's conjecture \cite{deninger}, recently proved by Rogers and Zudilin\cite{rogerszudilin}, that its Mahler measure is a value of an elliptic curve $L$-function), for which $P$ and $\Pyinline$ both vanish at $\bigl(\exp(\pm i \pi/3),-1\bigr)$.  We calculated $\Delta_{n}(P)$ for $1\le n\le 1600$.  Based on the numerical evidence, $\Delta_{n}(P)$ almost certainly has a main term of the form $c(n)/n^{3/2}$, where $c(n)$ depends only on $n\jmod{6}$.  This is consistent with the sort of bounds given by Boyd in \cite{boyd2}.  In fact, it appears that the difference $\Delta_{n}(P) - c(n)/n^{3/2}$ vanishes more quickly than $1/n^{r}$ for all $r>0$---that is, $\Delta_{n}(P) \quasi c(n)/n^{3/2}$---but the evidence is less clear for this.


\section{Algebraic determination of the signs $s(x,y)$}\label{sec:signdet}

As mentioned earlier, for a point $(\alpha,\beta)$ in $\mathcal{E}$, its sign $s(\alpha,\beta)$ can be determined algebraically, without reference to the root function $\rho(x)$ such that $\beta=\rho(\alpha)$.  Let $f(t)\deq \Log\bigl(\rho(\alpha \exp(it))/\beta\bigr)$, where $\Log$ is the principal branch of the logarithm.  The function $\Re\bigl(f(t)\bigr) =  \log\bigl|\rho(\alpha \exp(it))\bigr|$ experiences the same sign change at $0$ (with $t$ moving left to right in $\R$) as $\abs{\rho(x)}-1$ does at $\alpha$ (with $x$ moving counterclockwise along $\UC$).
We can determine arbitrarily many of the coefficients in the Maclaurin series for $f(t)$ with the use of \lemref{lem:implicitderivs}.  The first of these coefficients whose real part is nonzero will determine the behavior of $\Re\bigl(f(t)\bigr)$ at $t=0$, and hence determine $s(\alpha,\beta)$.

More specifically, say $f(t) = \sum_{k=1}^{\infty} b_{k} t^{k}$.  ($b_{0}=f(0)=0$.)  Let $N$ be the smallest positive integer such that $\Re(b_{N})\ne 0$.
Then
\[
s(\alpha, \beta) = \begin{cases}
0, 				& \text{if $N$ even,} \\
\sign(\Re(b_{N})),	& \text{if $N$ odd.}
\end{cases}
\]
In particular, $b_{1} = i \rho'(\alpha) = -i P_{x}/P_{y}$ (with $P_{x} = \partial P/\partial x$ and $P_{y} = \partial P/\partial y$ evaluated at $(\alpha,\beta)$).  Generically, $b_{1}$ will be nonzero.  Therefore, it will usually suffice to use the rule of thumb that $s(\alpha, \beta) = \sign(\Im(P_{x}/P_{y}))$, as long as the right side is nonzero.

Incidentally, for $k\ge 1$, it appears that
\[
b_{k} \stackrel{?}{=} \frac{i^{k}}{k!}\sum_{j=1}^{k} \stirtwo{k}{j} \rho^{(j)}(\alpha) \beta^{j-1},
\]
although we have not verified this identity.

\section{Conclusions}

It is unclear to what extent it is possible to remove the hypothesis on $P$ and $\Pyinline$ in \thmref{thm:main}.  If $Z(P)$ and $Z(\Pyinline)$ have a common solution in $\UC\times\C$, then a root function will have an algebraic singularity on $\UC$.  Our current approach would require us to integrate up to that singularity, but this causes havoc with our use of \lemref{lem:parts}.  We devoted extensive effort to circumventing this difficulty, for instance with Puiseux expansions and the method of stationary phase \cite[Section~2.9]{erdelyi}, but those attempts have not been fruitful.  The evidence from section~\ref{subsec:deninger} (including the fractional exponents, echoing those found in Puiseux expansions) might lead one to expect a.p.p.s.'s in powers of $n^{-1/M}$ for some positive integer $M$.
However, if similarly complicated formulas hold for the coefficients in such expansions, one would expect to have either all coefficients equal to 0 or infinitely many of them nonzero.  This is at odds with the evidence from the example in~\ref{subsec:deninger}.  This discrepency leads us to doubt that \thmref{thm:main} can be extended to all nonzero polynomials $P(x,y)$ without a significant modification of its statement.

\section{Acknowledgements}

I would like to thank Fernando Rodriguez Villegas for first suggesting this problem to me and for valuable suggestions.  I am also grateful to Rob Benedetto for providing extensive advice on an earlier draft of this paper, and to David Cox for helpful conversations.

I performed a great deal of computer experimentation in the course of this research.  I made extensive use of Sage \cite{sage471} and PARI/GP \cite{pari243} (both within Sage, and on its own).  I also used Mathematica \cite{mathematica7} for earlier work.


\bibliographystyle{elsarticle-num}
\bibliography{mybib}

\medskip

\end{document}